\documentclass[10pt, twocolumn, twoside, reqno]{IEEEtran}

\usepackage{amsmath, amsfonts}
\usepackage{amssymb}

\usepackage{latexsym}
\usepackage{graphicx}
\usepackage{mathrsfs}
\usepackage{color}
\usepackage{dsfont}
\usepackage[latin1]{inputenc}

\newcommand\NoBlackBoxes{\global\overfullrule0pt}
\NoBlackBoxes

\makeatletter
\let\serieslogo@\relax
\let\@setcopyright\relax

\newtheorem{definition}{Definition}[section]
\newtheorem{theorem}[definition]{Theorem}
\newtheorem{lemma}[definition]{Lemma}
\newtheorem{proposition}[definition]{Proposition}
\newtheorem{rem}[definition]{Remark}

\renewcommand{\epsilon}{\varepsilon}
\renewcommand{\phi}{\varphi}

\begin{document}

\setcounter{page}{1}

\title{On the capacity of a new model of  associative memory based on neural cliques}

\author{Judith Heusel,
\thanks{Fachbereich Mathematik und Informatik,
University of M\"unster,
Einsteinstra\ss e 62,
48149 M\"unster,
email: {\tt jheus01@uni-muenster.de}}
Matthias L\"owe,
\thanks{Institut f\"ur Mathematische Statistik,
Fachbereich Mathematik und Informatik, Universit\"at M\"unster,
Einsteinstr. 62, 48149 M\"unster, Germany, email: {\tt maloewe@math.uni-muenster.de}}
Franck Vermet
\thanks{Departement de Mathematiques, Universite de Bretagne Occidentale,  6, avenue Victor Le Gorgeu
B.P. 809
F- 29285 Brest Cedex, France, email: {\tt Franck.Vermet@univ-brest.fr}}
}


\date{\today}


\maketitle
%
%

\begin{abstract}
Based on recent work by Gripon and Berrou, we introduce a new model of an associative memory. We show that this model has an efficiency bounded away from 0 and is therefore significantly more effective than the well known Hopfield model. We prove  that the synchronous and asynchronous retrieval dynamics converge and give upper and lower bounds on the memory capacity of the model.
\end{abstract}


\section{Introduction}
In \cite{griponb} Gripon and Berrou introduced a new model of an associative memory. This model seems to be more effective than standard models of associative memories, in particular the Hopfield model, while at the same time it respects the basic principles of associative memories, e.g. locality. In this model, which we will call GB-model for short, there are
$N$ neurons grouped into $c$ groups of $l$ neurons. Typically, one $c$ is bounded from below by a constant and from above by $\log l$, where the latter bound seems to be the most interesting case. One tries to store $M$ messages $m^1,\ldots,m^M$ in this network. These messages are sparse in the sense that each message $m^\mu$ has $c$ active neurons, only, i.e.
$m^\mu=(m_1^\mu,\ldots,m_c^\mu)$, and for each $\mu=1, \ldots M$ and each $i=1, \ldots c$, $m_i^\mu$ denotes the (only) active neuron of the message $m_\mu$ in the i'th block.
With such a message $m^\mu$ one associates the edges of the complete graph $K^\mu$ spanned by the vertices $m_1^\mu,\ldots,m_c^\mu$.
A message $m^0=(m_1^0,\ldots,m_c^0)$ is considered to be stored in the model, if all edges of the complete graph spanned by $(m_1^0,\ldots,m_c^0)$ are present in the set of edges
$${\mathcal{E}}:=\{e: e \mbox{ is an edge of one of the } m^\mu\}.$$

Gripon and Berrou now analyze in a number of papers (see e.g. \cite{griponb},\cite{gripona}, \cite{griponc} among others) the performance of this network for a random input $m^1,\ldots,m^M$, e.g. they ask for the probability that a random pattern $m^0$ is stored or that a corrupted pattern can be retrieved. However, all their analysis is either based on numerical simulations or on the assumption that the events that a given edge occurs in ${\mathcal{E}}$ are independent. This is, of course, not true, even though the dependence is very weak. On the other hand, the precise form of the network makes it difficult to analyze it rigorously.
In the present paper we therefore strive for the rigorous analysis of yet another associative memory model, which is, however, closely related to the GB model.

To motivate it, observe that the GB model can be mathematically described in the following way:
We start with an alphabet ${\mathcal{A}}= \{1, \ldots, l\}$. A message $m^\mu$ is then a string $m^\mu=(m_1^\mu,\ldots,m_c^\mu)\in {\mathcal{A}}^c$.
With a message $m^\mu$  we associate a (column) vector $\psi(m^\mu)\in (\{0,1\}^l)^c$ obtained by replacing the $m_i^\mu$ with the unit vector $e_{m_i^\mu}$.
In a slight abuse of notation we will also use the notation ${\mathcal{A}}^c$ for the set $\{e_1,\ldots,e_l\}^c$.
Now build the
the 0-1-matrix
$W$ given by
$$
\tilde W = \max_{m \in {\mathcal{M}}} \psi(m) \psi(m)^t
$$
where ${\mathcal{M}}=\{m^1,\ldots,m^M\}$ and $\psi(m)^t$ is the transpose of $\psi(m)$. Thus
\begin{itemize}
\item for $a\not=a'$ we have $\tilde W_{(a,k),(a',k')} = 1$ if and only if there is an edge between $(a,k)$ and $(a',k')$.
\item for $k\not= k'$ we have $\tilde W_{(a,k),(a,k')} = 0$
\item and $\tilde W_{(a,k),(a,k)} = 1$ if and only if there exists $\mu$ such that the k'th neuron in block $a$ is 1 (this is equivalent
to adding a self-loop to the graph for each vertex $(a,k)$ such that there exists
$\mu$ with $m^\mu_a= k$.
\end{itemize}

With this matrix one can associate a dynamics $D$ on $(\{0,1\}^l)^c$ : for $v\in (\{0,1\}^l)^c$,
$$D(v)_{(a,k)}= \bigwedge_{b=1}^c \bigvee_{r=1}^l \tilde W_{(a,k), (b,r)} v_{(b,r)},$$
which can also be written as
$$\displaystyle D(v)_{(a,k)} =  \mbox{\bf{1}}_{\{\sum_{b=1}^c \mbox{\bf{1}}_{\{\sum_{r=1}^l \tilde W_{(a,k), (b,r)} v_{(b,r)} \ge 1\}}\ge c\}}.$$

Obviously, for all learned message $m\in {\mathcal {M}}$, we have
$$D(\psi(m))= \psi(m).$$
However, a more detailed analysis of the associative abilities of the network turn out to be difficult, basically due to the double indicator structure of $D$ and the function {\it max} in ${\tilde W}$.

We therefore propose the following variant of the above model. Consider the Bernoulli random variables $\zeta_{(a,i)}^\mu$, that denote if neuron $i$ of cluster $a$ is part of message $\mu$:
$$\zeta_{(a,i)}^\mu=\begin{cases} 1 & \text{ if }m_a^\mu=e_i,\\
0 & \text{ otherwise.}
\end{cases}
$$
Under the above assumptions on the patterns, i.e. their independence and equidistribution, the random variables $(\zeta_{(a,i)}^\mu)_{1\leq a\leq c,1\leq j\leq l}^{1\leq\mu\leq M}$ are Bernoulli variables, each with parameter $\frac{1}{l}$, and $\zeta_{(a,i)}^\mu$ is independent of $\zeta_{(b,j)}^\nu$ if $a\ne b$ or $\mu\ne \nu$.\\
We define the matrix $W\in{\mathbb N}^{cl\times cl}$ by
$$W_{(a,i),(b,j)}=\sum_{\mu=1}^M\zeta_{(a,i)}^\mu\zeta_{(b,j)}^\mu,
$$
for $a,b\in \{1,\ldots,c\}, a\ne b$ and $i,j\in \{1,\ldots,l\}$.

We set \vspace{2mm}$W_{(a,i),(a,j)}=0$ for all $i,j\in\{1,\ldots,l\}$ and $a\in \{1,\ldots,c\}.$ Hence, $W_{(a,i),(a,i)}=0$, while $\tilde W_{(a,i),(a,i)}\in \{0,1\}$. However, this choice is necessary due to the structure of $W$ as a sum or random variables.

Given an input vector $v=(v_{(b,j)})_{1\leq b\leq c,1\leq j\leq l}\in\{0,1\}^{cl}$, we define the dynamics
\begin{equation}\label{dyn}
\varphi_{(a,i)}(v)={\mathds 1}_{\lbrace\sum_{b=1}^c\sum_{j=1}^lW_{(a,i),(b,j)}v_{(b,j)}\geq\kappa c\rbrace}
\end{equation}
for some $\kappa>0$. It should be obvious that this variant is closely related to the GB model. On the other the structure of $W$ and the dynamics that include sums of random variables rather maxima and minima, makes it more accessible to probabilistic tools. Indeed from this point of view the model is reminiscent to the Hopfield model, whose the storage capacity has been analyzed in \cite{MPRV}, \cite{KomlosPaturi1988}, \cite{Newman_hopfield}, \cite{loukianova}, \cite{talagrand}, \cite{Lo98}, \cite{LV05},\cite{LV11}, and many other papers.\\

In the present note we want to justify our variant of the GB model by showing that it has a non-vanishing efficiency in the range of parameters we want to study. This will be done in Section 2. In Section 3 we introduce the sequential and parallel retrieval of the model and show that it converges, possibly to a cycle of length 2. Section 4 contains our bounds on the memory of storage capacity of our version of the GB model. Sections 5 and 6 are devoted to the proofs of these results.

\section{The efficiency of the model}
In this section we want to justify that our version of the GB model is potentially better than other models of neural networks, in particular the Hopfield model, that may be the best studied model of an associative memory. We will do this with a performance measure that is known as the {\it informational efficiency} $\eta$. $\eta$ compares the amount of information in the memorized messages to the logarithm of the number of bits needed to describe all the bonds in a state of the network at the storage capacity.

There are good reasons for this definition (and not comparing the amount of information in the memorized messages to the information in the network, e.g.). First, for any reasonable neural network the state of the network is a function of the information to be stored.
In particular, if we compared the information in the messages to the information in the network and we were in the retrieval regime $\eta$ would be 1 for any network. Second, the point in the entire theory of associative memories is that we do not know the stored information. Hence the best way to describe our network is to describe it bond by bond, as if they were independent (for a related discussion see also \cite{Jagota96}).

For our network, in view of Section 4, we will always choose $M=\alpha l^2$ and $c=\log l$.

Recall that in information theory a source is a probability distribution $\pi=(\pi_1, \ldots, \pi_A)$ on a (finite) alphabet $\{1, \ldots, A\}$ (for background material the reader is referred to classical textbook of Ash \cite{Ash}). Its information content is measured by the (binary) entropy
$$
H(\pi) = -\sum_{i=1}^A \pi_A \log_2 \pi_A.
$$
$H(\cdot)$ is maximized by the uniform distribution $u(i)=1/A$ for all i, in which case we obtain $H(u)=\log_2 A$. Assuming that we sample the $M$ words ${\mathcal{M}}$ with the uniform distribution  on all possible pairwise different $\binom{l^c}{M}$ we obtain that
$$
H({\mathcal{M}}) = \log_2 \left(\binom{l^c}{M}\right) \simeq Mc \log_2(l),
$$
for $l$ large, $M=\alpha l^2$, and $c=\log(l)$.

On the other hand, the network is completely described by the symmetric matrix $W$. There are $\binom{c}{2} l^2$ many matrix entries. Moreover for $l \to \infty$ a single matrix entry is asymptotically Poisson-distributed with parameter $\alpha$.

This implies that the entropy of the network satisfies
\begin{eqnarray*}
H(W)  &= &\binom{c}{2} l^2 H(W_{(1,1),(2,1)}) \\ &\simeq& \binom{c}{2} l^2\  h({\mathcal{P}}\mbox{ois}(\alpha)),
\end{eqnarray*}
where
\begin{eqnarray*}
h({\mathcal{P}}\mbox{ois}(\alpha))&=& \frac 1{\log(2)}(\alpha(1-\log(\alpha))\\
&&\quad +e^{-\alpha}\sum_{k=0}^{+\infty} \alpha^k \frac{\log(k!)}{k!} )
\end{eqnarray*}
is the entropy of the Poisson law with parameter $\alpha$.

Therefore we obtain
$$
\eta = \frac{ H({\mathcal{M}}) } {H(W)} \simeq  \frac {2 \alpha }{\log(2)\ h({\mathcal{P}}\mbox{ois}(\alpha))}  >0.
$$
Note that in the Hopfield model with $M= \mathrm{const.} N /\log N$ patterns the efficiency tends to 0 when $N$ is large,
while the efficiency for our model is bounded away from zero.
Indeed, for the Hopfield model, $H({\mathcal{M}})$ is of order $N^2/\log(N)$, while $H(W)$ is larger than $N^2$, since the matrix $W$ has $N(N-1)/2$ different entries and each entry $W_{ij}$ follows the Binomial law ${\mathcal{B}}in(M, 1/2)$.

Moreover also notice, that $\eta$ is increasing with $\alpha$, and one can compute numerically that $\eta > 1$ if $\alpha \geq 0.423$. Interestingly, in
Theorem \ref{thm3} we obtain similar critical value for $\alpha$: Both results show that if $\alpha$ is too large, the system looses information in the sense that $H({\mathcal{M}}) > {H(W)} $ and not all the information contained in the words can be stored in the system.

\section{The retrieval dynamics}
Already in \eqref{dyn} we introduced the dynamics $\phi$ of our version of the GB model.

As a matter of fact, one can associate either a sequential or a parallel dynamics with $\varphi$.  In this section will prove that for all initial condition, both dynamics converge.

First let
$$S= \varphi_{(c,l)}\circ \varphi_{(c,l-1)} \circ \ldots \circ \varphi_{(1,1)}.$$ For an arbitrary initial state $v(0)$ we define the sequential dynamics $v(t+1)=S(v(t))$.

Then the following result holds:
\begin{proposition}\label{lemma1}
Let $H=H_S$ be the following function:
$$H_S(v) = -\frac 12 \sum_{(a,i)} \sum_{(b,j)}  W_{(a,i), (b,j)}  v_{(a,i)} v_{(b,j)} +  \kappa c \sum_{(a,i)}  v_{(a,i)}.$$
The reason, why it carries an index $S$ is, that with $H$ as Hamiltonian (or energy function), the dynamics $S$ is a Hamiltonian zero temperature (or gradient descent) dynamics.

Then for all $M,c,l$, and $v(0)$,
$$H_S(v(t+1)) \le H_S(v(t)) \qquad \mbox{for all $t$,}$$
and the sequential dynamics converges to a fixed point.
\end{proposition}

\begin{proof}
To simplify notation, we denote by capital letters the double indices, for example: $I= (a,i)$.
With this notation, $H_S$ can be written as
$$H_S(v) = -\frac 12 \sum_{I} \sum_{J}  W_{I, J}  v_{I} v_{J} +  \kappa c \sum_{I}  v_{I}.$$

Let $v\in (\{0,1\}^l)^c$ not be a fixed point of $S$. Then there exists $K$ such that
 $S_K(v):= (S(v))_K \ne v_K$. Then we can define $w\in (\{0,1\}^l)^c$ such that $w_K=S_K(v)$ and  $w_I= v_I$ for all $I\ne K$. Then we need to prove that
$$H_S(w ) \le H_S(v).$$

First suppose that $v_K=0$ and $w_K=S_K(v)=1$. Then
\begin{eqnarray*}
&& H_S(w) - H_S(v)\\&= &-\frac 12 \sum_{I(\ne K)}  W_{I, K}  v_{I} -\frac 12 \sum_{J(\ne K)}  W_{K, J}  v_{J}\\
&&\qquad + \kappa c \\
& =& -\sum_{I}  W_{I, K}  v_{I} +\kappa c\\
&\le& 0
\end{eqnarray*}

Here we first use the symmetry $W_{I,J}= W_{J,I}$,
and we exploit $W_{K,K} = 0$.
Then finally, the last inequality  is true since $S_K(v)=1$ if and only if $\sum_{I}  W_{I, K}  v_{I} -\kappa c \ge 0$.

\vspace{2mm}
The second case is very similar : if $v_k=1$ and $w_K=S_K(v)=0$, then
\begin{eqnarray*}
&&H_S(w) - H_S(v)\\&= &\frac 12 \sum_{I(\ne K)}  W_{I, K}  v_{I} \\
&&\quad +\frac 12 \sum_{J(\ne K)}  W_{K, J}  v_{J} -\kappa c \\
& = &\sum_{I}  W_{I, K}  v_{I} -\kappa c\\
&<& 0
\end{eqnarray*}

Here we use that $W_{K,K} = 0, v_K=1$ and the last inequality is true since $S_K(v)=0$ if and only if $\sum_{I}  W_{I, K}  v_{I} -\kappa c < 0$.

We deduce that $H_S(v(t+1)) \le H_S(v(t))$ for all $t$ and that for all initial state $v(0)$, the dynamics converges to configurations minimizing at least locally the Hamiltonian. Moreover, the dynamics converges finally to a  fixed point, that is $v$ such that $S(v)=v$, since flipping from 1 to 0 strictly decreases the Hamiltonian and there are only finitely many possible energies.
\end{proof}

We will now consider the parallel dynamics.
Let $T= (\varphi_{(1,1)},  \ldots, \varphi_{(l,c)}),$ and for  all $v(0)$, the parallel dynamics $v(t+1)=T(v(t))$, i.e.
for all $(a,i),  v_{(a,i)}(t+1)= \varphi_{(a,i)}(v(t))$.
We have the following result :

\begin{proposition}\label{lemma2}
Let $H=H_T$ be the following function:
\begin{eqnarray*}
H_T(v) &=& - \sum_{(a,i)} \sum_{(b,j)}  W_{(a,i), (b,j)}  v_{(a,i)}\  y_{(b,j)} \\&& \quad +  \kappa c \sum_{(a,i)}  (v_{(a,i)} + y_{(a,i)}),
\end{eqnarray*}
where $y=T(v)$. Again $H$ can be considered as the Hamiltonian that turns $T$ into a Hamiltonian dynamics.

Then for all $M,c,l$, $v(0)$,
$$
H_T(v(t+1)) \le H_T(v(t)) \qquad \mbox  {for all $t$},$$
and the parallel dynamics converges to a fixed point $v^*$, i.e. $T(v^*)=v^*$, or to a limit cycle $(\tilde{v}^1, \tilde{v}^2)$ of length 2, i.e.
$$
T(\tilde{v}^1)=\tilde{v}^2 \quad \mbox{and   } T(\tilde{v}^2)=\tilde{v}^1.
$$
\end{proposition}

\begin{proof}
With the previous notations, $H_T$ can be written as
\begin{eqnarray*}
H_T(v(t)) &=& - \sum_{I} \sum_{J}  W_{I, J}  v_{I}(t)\ v_{J}(t+1) \\
&&\quad +  \kappa c \sum_{I}  (v_{I}(t) +v_{I}(t+1))
\end{eqnarray*}
and
\begin{eqnarray*}
H_T(v(t+1))- H_T(v(t)) &=& \\
&&\hskip-4.5cm - \sum_{J}  (v_{J}(t+2)-v_J(t))  (\sum_I W_{J, I}  v_{I}(t+1) -  \kappa c),
\end{eqnarray*}
since $W$ is symmetric. And thus, $$H_T(v(t+1))- H_T(v(t))\leq 0.$$
Indeed,
\begin{itemize}
\item if $\sum_I W_{J, I}  v_{I}(t+1) -  \kappa c\ge 0$, then $v_{J}(t+2)=1$, and $v_{J}(t+2)-v_J(t)\ge 0$,
\item if $\sum_I W_{J, I}  v_{I}(t+1) -  \kappa c< 0$, then $v_{J}(t+2)=0$, and $v_{J}(t+2)-v_J(t) \le 0$.
\end{itemize}

This proves that $H_T$ is decreasing along each trajectory of the parallel dynamics and this dynamics converges to a subset $V$ such that $H_T(T(v))- H_T(v)=0$, for all $v\in V$. Moreover, if we consider the 2-steps dynamics $v(t+2)= T\circ T(v(t))$, the previous calculations show that, there exists a $t_0$ such that for all $t\ge t_0$,  for all $J$,
\begin{itemize}
\item either $v_J(t+2)= v_J(t),$
\item or  $v_J(t+2)\ne v_J(t)$ and $$\sum_I W_{J, I}  v_{I}(t+1) -  \kappa c=0,$$
which implies $v_J(t)=0, v_J(t+2)=1$.
\end{itemize}
Thus, after a finite number of steps, we have $v_I(t+2)= v_I(t),$ for all $I$, i.e. the dynamics converges to a fixed point or to a 2-cycle.

\end{proof}

\section{Bounds on the storage capacity}

The purpose of the present section is to analyze whether one can find threshold $\kappa$ such that an amount of $M=\alpha l^2$ patterns can be stored in the above described variant of the GB model. As we will see that this is indeed the case: Also in this sense our model is considerably more effective than the Hopfield model. Recall that in the latter, only for $M \le N/(2 \log N)$ all patterns are fixed points of the retrieval dynamics (see \cite{MPRV}, \cite{Burshtein},\cite{Bov98},and \cite{L98}, \cite{LV05} for correlated information). In our variant of the GP model there are $N= l \log l$ neurons, while the number of messages to be stored is proportional to $l^2$.
More precisely, we will prove the following theorem.

\begin{theorem}\label{thm1}
In the above model with coding matrix $W$, let $c=\log l$ and $M=\alpha l^2$. For $\kappa \le 1-1/c$ we have
\begin{enumerate}
\item If $\alpha < \kappa$,  for every fixed $\mu$,  every fixed block $a$ and every fixed coordinate $k$, we have that
$$
{\mathbb{P}}\left(\varphi_{(a,k)}(m^\mu)\neq m^\mu_{(a,k)}\right) \to 0
$$
as $l$ and therefore $N$ tends to infinity.
\item 
If $\alpha <  \kappa \exp(-(3+\kappa)/\kappa)$,
we have  that
{{
\begin{eqnarray*}
 {\mathbb{P}}\left(\exists \mu\in\{1,\ldots, M\}: T(m^\mu)\neq m^\mu\right) =\\
{\mathbb{P}}\left(\exists a, \exists i,\exists \mu: \varphi_{(a,i)}(m^\mu)\neq m^\mu_{(a,i)}\right) \to 0
\end{eqnarray*}}}
as $l$ and therefore $N$ tends to infinity.
\end{enumerate}
\end{theorem}

\begin{rem}\normalfont
Again, the comparison to the Hopfield model is interesting. As mentioned above the storage capacity with $N$ neurons there is $N/(2 \log N)$. However, note that the information in our case is extremely sparse. Indeed, only $c=\log l$ out of the $l \log l$ many entries of a message are non-zero. A similar situation for the Hopfield model was considered in the non-rigorous paper by Amari \cite{Amari89}. He asserts that for a Hopfield model with $N$ neurons, where each message has ${\mathcal{O}}(\log N)$ many 1's and the other neurons are 0, the storage capacity is of order $N^2/(\log N)^2$, which is only slightly worse than our result.  \end{rem}

Moreover, we will also be interested in the error correcting abilities of the model.
Such  errors in a message occur,
 if some characters are false, or  erased. In fact, both types of error are  equivalent, if we replace each  missing character with a randomly chosen letter.
If we not only require the messages to be fixed points of the network dynamics, but also that the network is able to correct a certain percentage of errors (hence works as a truly associative memory), this may lower its capacity. However, the order of the capacity is maintained as can be read off from the following theorem. We  will concentrate on one step of the parallel dynamics.
To this end we define the discrete ball of radius $r$ centered in $m^\mu$ as
\begin{eqnarray*}
{\mathcal{B}}(m^\mu, r) &=& \{ m \in {\mathcal{A}}^c :  d_H(m^\mu, m)\\ &&= \mbox{card}\{ j : m^\mu_j\neq m_j\} \le r\}.
\end{eqnarray*}

With this notation we have the following result.

\begin{theorem}\label{thm2}
In the above model with coding matrix $W$, let $c=\log l$ and $M=\alpha l^2$. Moreover choose $\gamma\in ]0,1[$,
and set $ \kappa = \min\{ 1-\gamma, 1-1/c\}, $
and take $\alpha < \kappa \exp(- (1+\kappa)/\kappa)).$

Then for all $\mu=1,\ldots, M$ and $v$ randomly chosen in ${\mathcal{B}}(m^\mu, \gamma c - 1)$, we have

$$
{\mathbb{P}}\left( T(v)= m^\mu\right) \to 1
$$
as $l$ and therefore $N$ tends to infinity.

In other words for $\alpha >0$ small enough we can correct any randomly chosen input with at most $\gamma c-1$ errors.
\end{theorem}

\begin{rem}
\normalfont
Observe that, other than in the Hopfield model, our model can even repair more than 50 percent mistakes. This is due to the fact  that in the Hopfield model the situation between inputs $+1$ and $-1$ is symmetric, while here a signal $0$ is something completely different than a signal $1$. Repairing wrong input messages is also a problem in the original GB model, since there the number of $1$'s is decreasing with the dynamics, i.e. a wrong 0 will always stay a wrong 0.

Let us remark that the  Theorem \ref{thm2} is true for random errors and can't be extended to all messages in ${\mathcal{B}}(m^\mu, \gamma c - 1)$, for $\gamma \in ]0,1[$. For example, if we consider the message $v$, which coincides with the message $m^1$ for the blocks $a=1,\ldots, c/2$ and  with $m^2$ for the block $c/2+1,\ldots, c$, we have necessarily $T(v)\ne m^1$ or $T(v)\ne m^2$. The correction of all possible corrupted messages with a given number of errors is a really different problem: as mentioned in \cite{KomlosPaturi1988} for the Hopfield model, random errors and worst case errors behave
in entirely different ways.
\end{rem}

For the Hopfield model, using negative association,  Bovier proved  in \cite{Bov98} that the capacity $M=N / (2 \log N)$ is in a certain sense optimal and cannot be surpassed. For our model,
 Theorem \ref{thm1} states that for $M= \alpha l^2$ , with $\alpha$ small enough, the memorized messages are fixed points of the dynamics: From the proof of this theorem in the following section, we see that if $\kappa\le 1-1/c$ and $\alpha \le  \kappa \exp(-(1+\kappa)/\kappa)$,  then for all $\mu=1,\ldots, M$,
$$\lim_{l\rightarrow\infty}{\mathbb{P}}(T(m^\mu)= m^\mu)=1.$$
In particular, if we choose $\kappa= 1-1/c$, we obtain the maximal value
$$\alpha^* =(1-1/c)\exp(-1-\frac{c}{c-1}) \simeq e^{-2}\simeq 0.135.$$
A natural question is whether the bound we get for $\alpha$ is optimal.
A first step in this direction is the following result where we obtain an upper bound for the capacity, in the sense of the stability of the memorized messages.

\begin{theorem}\label{thm3}
Let $\kappa\leq 1-1/c$. Then, for each $\alpha>-\log(1-e^{-1})\simeq 0.45$ and each $\mu\in\lbrace1,\ldots,M\rbrace$, the message $m^\mu$ is not stable with probability converging to 1 when $l$ tends to infinity, i.e.
$$\lim_{l\rightarrow\infty}{\mathbb{P}}(T(m^\mu)\neq m^\mu)=1.
$$
\end{theorem}

\begin{rem}
\normalfont
All the results concerning the stability of stored messages are the same if we consider the sequential dynamics $S$ instead of the parallel dynamics $T$.
\end{rem}

\section{Proof of Theorems \ref{thm1} and \ref{thm2}}
In this section Theorems \ref{thm1} and \ref{thm2} are proved. The proofs are rather similar and employ exponential inequalities. We will treat the proof of Theorem \ref{thm1} in greater details and will be slightly more sketchy about the proof of Theorem \ref{thm2}.

\begin{proof}[Proof of Theorem \ref{thm1}]
For a learned message $m^\mu \in {\mathcal{M}}$ we have to check, whether $m^\mu$ is a fixed point of the dynamics $\varphi$. Without loss of generality $\mu=1$ and the first pattern has a one entry at the first position of every block and $0$'s otherwise. That is to say $m^1=(e_1,\ldots,e_1)$.

Now there are two cases to consider, depending on whether we want to check that a $1$ remains a $1$ under the dynamics or a $0$ remains a $0$. Let us start with the first case.

For $i=1$, we have for $1\leq a\leq c$
$$\sum_{b=1}^c\sum_{j=1}^lW_{(a,1),(b,j)}m^1_{(b,j)}= \sum_{b=1}^c W_{(a,1),(b,1)} \ge c-1,
$$
so if we choose $\kappa \leq 1-1/c$, we immediately see that
$${\mathbb{P}}(\varphi_{(a,i)}(v)=1)=1.$$
So all entries with a $1$ are trivially fixed points of our dynamics.

For the case that the entry of $m^1$ is zero, we may without loss of generality consider
the case $a=1$ and $i=2$. Then
\begin{eqnarray}
&&{\mathbb{P}}\left(\varphi_{(1,2)}(m^1)\neq m^1_{(1,2)}\right)\nonumber\\
&=&{\mathbb{P}}\left(\sum_{b=2}^c\sum_{j=1}^lW_{(1,2),(b,j)}m^1_{(b,j)}\geq\kappa c\right)\nonumber\\
&=&{\mathbb{P}}\left(\sum_{b=2}^c\sum_{j=1}^l\sum_{\mu=2}^M\zeta_{(1,2)}^\mu\zeta_{(b,j)}^\mu m_{(b,j)}^1\geq \kappa c\right)\nonumber\\
&=&{\mathbb{P}}\left(\sum_{b=2}^c\sum_{\mu=2}^M\zeta_{(1,2)}^\mu\zeta_{(b,1)}^\mu \geq \kappa c\right)\nonumber\\
&\leq&e^{-t\kappa c}{\mathbb{E}}\left[e^{t\sum_{b=2}^c\sum_{\mu=2}^M\zeta_{(1,2)}^\mu\zeta_{(b,1)}^\mu }\right]\nonumber\\
&=&e^{-t\kappa c}{\mathbb{E}}\left[e^{t\sum_{b=2}^c\zeta_{(1,2)}^2\zeta_{(b,1)}^2 }\right]^{M-1}\label{first_estimate}
\end{eqnarray}
for some $t>0$.
Here we just apply the definition of the dynamics, the definition of $W$, the fact that we know, for which coordinates $m^1$ has a non-zero entry, and finally an exponential Chebyshev-inequality and the independence of the messages.

We compute ${\mathbb{E}}\left[e^{t\sum_{b=2}^c\zeta_{(1,2)}^2\zeta_{(b,1)}^2 }\right]$ and obtain the following bound:
\begin{eqnarray*}
&&{\mathbb{E}}\left[e^{t\sum_{b=2}^c\zeta_{(1,2)}^2\zeta_{(b,1)}^2 }\right]\\
&=&\left(1-\frac{1}{l}\right)\cdot1+\frac{1}{l}\left({\mathbb{E}}[e^{t\sum_{b=2}^c\zeta_{(b,1)}^2}]\right)\\
&\leq&\left(1-\frac{1}{l}\right)+\frac{1}{l}\left(\frac{e^t-1}{l}+1\right)^c\\
&\leq&\left(1-\frac{1}{l}\right)+\frac{1}{l}e^{c\cdot\frac{e^t-1}{l}},
\end{eqnarray*}
where we used the standard estimate $1+x \le e^x$ for all $x$.
Plugging this into \eqref{first_estimate}, we arrive at
\begin{eqnarray}
&&{\mathbb{P}}\left(\varphi_{(1,2)}(m^1)\neq m^1_{(1,2)}\right)\\
&\leq&e^{-t\kappa c}\left[\left(1-\frac{1}{l}\right)+\frac{1}{l}e^{c\cdot\frac{e^t-1}{l}}\right]^M\nonumber\\
&\leq&e^{-t\kappa c}e^{\frac{M}{l}\left(e^{c\cdot\frac{e^t-1}{l}}-1\right)}\nonumber.
\end{eqnarray}

Now, for fixed $t$ and $l\rightarrow\infty$, we have by expanding the first exponential in the exponent
\begin{eqnarray*}
&&{\mathbb{P}}\left(\varphi_{(1,2)}(m^1)\neq m^1_{(1,2)}\right)\\
&\leq&e^{-t\kappa c}\exp\left(\frac{M}{l}\left(e^{c\cdot\frac{e^t-1}{l}}-1\right)\right)\\
&=&e^{-t\kappa c}e^{\frac{M}{l}\left(c\frac{e^t-1}{l}+\frac 12\left(c\frac{e^t-1}{l}\right)^2+O\left(\left(c\frac{e^t-1}{l}\right)^3\right)\right)}\\
&=&e^{-t\kappa c}e^{\frac{Mc}{l^2}(e^t-1)+\frac{Mc^2}{2l^3}(e^t-1)^2+\frac{M}{l}O\left(\left(\frac{c}{l}(e^t-1)\right)^3\right)}\\
&\approx&\exp\left(c\left(-t\kappa+ \alpha(e^t-1)\right)\right),
\end{eqnarray*}
where for two sequences $(a_l)$ and $(b_l)$ we write $a_l \approx b_l$, if the fraction $a_l/b_l$ tends to 1 as $l$ tends to infinity.

The right hand side above takes its minimum at $t=\log(\kappa/\alpha)$. Plugging this value in  yields
$$
{\mathbb{P}}\left(\varphi_{(1,2)}(m^1)\neq m^1_{(1,2)}\right)\leq l^{\kappa-\alpha-\kappa\log(\kappa/ \alpha)}.
$$
If now $\alpha < \kappa$ the exponent is negative. This proofs the first assertion of the Theorem.

Moreover,
\begin{eqnarray*}
{\mathbb{P}}\left(\exists1\leq a\leq c,1\leq i\leq l,1\leq\mu\leq M:\varphi_{(a,i)}(m^\mu)\neq m^\mu_{(a,i)}\right)&&\\&&\hskip-8cm\leq Mlc\,e^{c\left(-t\kappa+ \alpha(e^t-1)\right)}
\end{eqnarray*}
Due to our choice of $M=\alpha l^2$ and $c=\log l$ this converges to 0, if
$$-t\kappa+ \alpha(e^t-1) < -3.$$
If we  replace $t$ by $\log(\kappa/\alpha)$ we obtain
$$-\kappa\log\left(\frac{\kappa}{\alpha}\right)+\kappa-\alpha<-3.
$$
We deduce that the previous inequality is true, if
$$-\kappa\log\left(\frac{\kappa}{\alpha}\right)<-3-\kappa,$$
that is $\displaystyle \alpha <  \kappa \exp(-(3+\kappa)/\kappa).$

\end{proof}

We will now continue with proof of Theorem \ref{thm2}.

\begin{proof}[Proof of Theorem \ref{thm2}]
Choose $0< \gamma <1$  and as indicated in the theorem, we set $\kappa=1-\gamma$. Suppose we want to correct a message  with $r$ errors, where $r \le\gamma c-1$.
 Again without loss of generality we concentrate on the case $\mu=1$ and $m^1=(e_1,\ldots,e_1)$.
Let $v=(e_2,\ldots,e_2,e_1,\ldots,e_1)$ be the vector consisting of $e_2$ in the first $r$ places and $e_1$ in the last $c-r$ places, i.e. $v$ is $m^1$ with $r$ errors.

Now there are now four different cases to differentiate:
\begin{itemize}
\item The false $1$'s in the clusters with wrong entries have to be turned into a $0$,
\item The correct $0$'s in all clusters have to remain $0$s,
\item The wrong $0$'s have to be turned into $1$'s,
\item and the correct $1$'s have to remain $1$'s.
\end{itemize}

The last two cases indeed are similar and with our choice of $\kappa$, they have probability one. Indeed, for example take $a=1$, $b=1$. Using $v_{(b,j)}=1$ only if $b\le r, j=2$, or $b\ge r+1, j=1$, we obtain
\begin{eqnarray*}
&&{\mathbb{P}}\left(\sum_{b=1}^c\sum_{j=1}^lW_{(1,1),(b,j)}v_{(b,j)}<\kappa c\right)\\
&=&{\mathbb{P}}\left(\sum_{b=1}^{r}\sum_{j=1}^lW_{(1,1),(b,j)}v_{(b,j)}\right. \\
&&\hskip2cm \left. < \kappa c-\sum_{b=r+1}^{c}W_{(1,1),(b,1)}\right)\\
&\le&{\mathbb{P}}\left(\sum_{b=1}^{r}\sum_{j=1}^lW_{(1,1),(b,j)}v_{(b,j)}<\kappa c-c+r\right)\\&=&0,
\end{eqnarray*}
since $\displaystyle\sum_{b=r+1}^{c}W_{(1,1),(b,1)}\ge c-r$ and $\kappa c-c+r\le -\gamma c +r\le0$.\\

This proves that a wrong $0$ becomes a $1$ with probability $1$. Let us now consider the case of the correct $1$'s. For example, for $c=1$ and $b=1$, we obtain

\begin{eqnarray*}
&&{\mathbb{P}}\left(\sum_{b=1}^c\sum_{j=1}^lW_{(c,1),(b,j)}v_{(b,j)}<\kappa c\right)\\
&=&{\mathbb{P}}\left(\sum_{b=1}^{r}\sum_{j=1}^lW_{(c,1),(b,j)}v_{(b,j)}\right. \\
&& \hskip 2cm \left. <\kappa c-\sum_{b=r+1}^{c}W_{(c,1),(b,1)}\right)\\
&\le&{\mathbb{P}}\left(\sum_{b=1}^{r}\sum_{j=1}^lW_{(1,1),(b,j)}v_{(b,j)}\right.\\
&& \hskip 2cm \left. <\kappa c-c+r+1\right)\\&=&0,
\end{eqnarray*}
since $\displaystyle\sum_{b=r+1}^{c}W_{(c,1),(b,1)}\ge c-r-1$ and $\kappa c-c+r-1\le -\gamma c +r-1\le0$.\\

The first two cases are also identical and can be treated analogously to our last calculations: For the first case and $a=1$, using $\zeta_{(1,2)}^1=0$, we get exemplary
\begin{eqnarray*}
&&{\mathbb{P}}\left(\sum_{b=1}^c\sum_{j=1}^lW_{(1,2),(b,j)}v_{(b,j)}\geq\kappa c\right)\\
&=&{\mathbb{P}}\left(\sum_{b=2}^c\sum_{j=1}^l\sum_{\mu=2}^M\zeta_{(1,2)}^\mu\zeta_{(b,j)}^\mu v_{(b,j)}\geq\kappa c\right)\\
&=&{\mathbb{P}}\left(\sum_{b=2}^{r}\sum_{\mu=2}^M\zeta_{(1,2)}^\mu\zeta_{(b,2)}^\mu+\sum_{b=r+1}^{c}\sum_{\mu=2}^M\zeta_{(1,2)}^\mu\zeta_{(b,1)}^\mu \geq\kappa c\right)\\
&\leq&e^{-t\kappa c}{\mathbb{E}}\left[e^{t\sum_{b=2}^{c}\sum_{\mu=2}^M\zeta_{(1,2)}^\mu\zeta_{(b,1)}^\mu}\right]\\
&\leq&e^{c\left(-t\kappa+\alpha(e^t-1)\right)}
\end{eqnarray*}
This function is minimal in $t=\log\left({\kappa}/{\alpha}\right)$.\\

So, if we want to correct a randomly chosen message $v\in {\mathcal{B}}(m^\mu, \gamma c -1)$, we need to bound the following probability
\begin{align*}
{\mathbb{P}}\left(T(v)\neq m^\mu\right)
\leq& \ cl e^{c\left(-t \kappa+\alpha(e^t-1)\right)}\\
\leq& \ e^{c\left(-t \kappa+\alpha(e^t-1)+1+o(1)\right)},
\end{align*}

We want that the right hand side converges to 0, which is the case if
$$-t \kappa+\alpha(e^t-1)<-1.
$$
After filling in the minimizing $t$, i.e. $t=\log\left(\kappa/{\alpha}\right)$, this becomes
$$-\log\left(\frac{\kappa}{\alpha}\right)\kappa+\kappa-\alpha<-1,
$$
which is true if
$$\alpha<\kappa \exp({-\frac{1+\kappa}{\kappa}}),
$$
with $\kappa=1-\gamma$. This finishes the proof of Theorem \ref{thm2}.
\end{proof}

\section{Proof of Theorem \ref{thm3}}

In this section, we prove the lower bound on the storage capacity. i.e. Theorem \ref{thm3}.

\begin{proof}[Proof of Theorem \ref{thm3}]
Again, without loss of generality, we take $\mu=1$ and $m^\mu=(e_1,\ldots,e_1)$
Observe that the message $m^1$  cannot be a fixed point of the dynamics, if there is at least one node $v_{(a,i)}$, $i\neq 1$, with $\phi_{(a,i)}(m^1)=1$, which holds if
$$\sum_{b=1}^cW_{(a,i),(b,1)}v_{(b,1)}\geq c-1,
$$
and
\begin{eqnarray*}
{\mathbb{P}}(T(m^1)\neq m^1)&=&\\
&&\hskip-4cm{\mathbb{P}}(\exists a\in\lbrace1,\ldots,c\rbrace,\exists i\in\lbrace2,\ldots,l\rbrace:\phi_{(a,i)}(m^1)=1).
\end{eqnarray*}

This is in particular the case if $v_{(a,i)}$ is connected to every neuron of $m^1$,
\begin{eqnarray*}
&&{\mathbb{P}}(T(m^1)\neq m^1)\\
&\geq&{\mathbb{P}}(\exists a\in\lbrace1,\ldots,c\rbrace,\exists i\in\lbrace2,\ldots,l\rbrace,\\
&&\quad  \forall b\in\lbrace1,\ldots,c\rbrace\setminus\lbrace a\rbrace: W_{(a,i),(b,1)}\geq1)
.
\end{eqnarray*}
This probability can be bounded by
\begin{eqnarray*}
&&{\mathbb{P}}(T(m^1)\neq m^1)\\
&\geq&{\mathbb{P}}(\exists a\in\lbrace1,\ldots,c\rbrace,\exists i\in\lbrace2,\ldots,l\rbrace, \\
&& \quad \forall b\in\lbrace1,\ldots,c\rbrace\setminus\lbrace a\rbrace: W_{(a,i),(b,1)}\geq1)\\
&=&1-{\mathbb{P}}(\forall  a\in\lbrace1,\ldots,c\rbrace,\forall i\in\lbrace2,\ldots,l\rbrace \\
&&\quad \exists b\in\lbrace1,\ldots,c\rbrace\setminus\lbrace a\rbrace: W_{(a,i),(b,1)}<1)\\
&\geq& 1-{\mathbb{P}}(\forall i\in\lbrace2,\ldots,l\rbrace \exists b\in\lbrace2,\ldots,c\rbrace:\\
&&\qquad  W_{(1,i),(b,1)}<1).
\end{eqnarray*}

The neurons $(1,i)$ and $(1,j)$, $i\neq j$, cannot be part of the same message. Hence, for $i\neq j$, we have
\begin{eqnarray*}
&&{\mathbb{P}}(\exists b_1, b_2\in\lbrace2,\ldots,c\rbrace:\\
&&\qquad W_{(1,i),(b_1,1)}<1,W_{(1,j),(b_2,1)}<1)\\
&\leq&{\mathbb{P}}(\exists b_1\in\lbrace2,\ldots,c\rbrace:W_{(1,i),(b_1,1)}<1)\times \\
&& \quad {\mathbb{P}}(\exists b_2\in\lbrace2,\ldots,c\rbrace: W_{(1,j),(b_2,1)}<1).
\end{eqnarray*}
Instead of an exact formal proof we give a more intuitive justification: If $(1,i)$ is not connected to every neuron $m^1_a$ of the other clusters a, $2,\leq a\leq c$, it will be more likely (but definitely not less likely, this is all we need) that $(1,j)$ is connected to every neuron $m_a^1$ of the other clusters, because $(1,i)\text{ and }(1,j)$ cannot be part of the same message. This means on the one hand that the fact that $(1,i)$ is not connected to every other cluster does not give any information about the neurons $m_a^1$ being activated in the same messages than $(1,j)$. On the other hand, $(1,i)$ and $(1,j)$ share the first cluster's $M-1$ places in the $M-1$ remaining messages, so they compete for the places, and the fact that $(1,i)$ is not connected to each of the other clusters indicates less messages containing $(1,i)$ and potentially more that can contain $(1,j)$. This argument can in fact be turned into a rigorous proof.\\
Altogether we have
\begin{eqnarray*}
&&\hskip -1.2cm {\mathbb{P}}(\forall i\in\lbrace2,\ldots,l\rbrace \exists b\in\lbrace2,\ldots,c\rbrace: W_{(1,i),(b,1)}<1)\\
&\leq&\prod_{i=2}^l{\mathbb{P}}(\exists b\in\lbrace2,\ldots,c\rbrace: W_{(1,i),(b,1)}<1)\\
&=&{\mathbb{P}}(\exists b\in\lbrace2,\ldots,c\rbrace: W_{(1,2),(b,1)}<1)^l.
\end{eqnarray*}
Thus, we have to estimate
$${\mathbb{P}}(\exists b\in\lbrace2,\ldots,c\rbrace: W_{(1,2),(b,1)}<1).
$$
\begin{lemma}\label{lemma3}
Let $m^1=(e_1,\ldots,e_1)$, \\$a\in\lbrace1,\ldots,c\rbrace$ and $i\in\lbrace2,\ldots,l\rbrace$.\\
We define $Y$ by the number of neurons belonging to the message $m^1$ which are connected to $(a,i)$.
Then for the distribution of $Y$, it holds that
$$
{\mathbb{P}}(Y=i)=\binom{c-1}{i}\left(1-e^{-\alpha}\right)^{i}e^{-\alpha(c-1-i)}(1+o(1)) 
$$
for $i\in\lbrace0,\ldots, c-1\rbrace$,
and 0, otherwise. 
In particular,
\begin{eqnarray*}
&&{\mathbb{P}}(\forall b\in\lbrace1,\ldots,c\rbrace\setminus\lbrace a\rbrace: W_{(a,i),(b,1)}=1)\\
&=&1-\left(1-e^{-\alpha}\right)^{c-1}(1+o(1))
\end{eqnarray*}
as $l\rightarrow\infty$.
\end{lemma}
With Lemma \ref{lemma3} in hands we can conclude the Proof of Theorem \ref{thm3} in the following way:
\begin{eqnarray*}
&&{\mathbb{P}}(\exists b\in\lbrace2,\ldots,c\rbrace: W_{(1,2),(b,1)}<1)\\&=&1-{\mathbb{P}}(\forall b\in\lbrace2,\ldots,c\rbrace: W_{(1,2),(b,1)}=1)\\
&=&1-\left(1-e^{-\alpha}\right)^{c-1}(1+o(1)).
\end{eqnarray*}
Finally,

$
{\mathbb{P}}(T(m^1)\neq m^1)
$

\vspace{-0.5cm}
$
\geq1-\left(1-\left(1-e^{-\alpha}\right)^{c-1}(1+o(1))\right)^l\stackrel{l\rightarrow\infty}{\rightarrow}1
$

for $\alpha>-\log(1-e^{-1})\approx0.45$.
\end{proof}

It remains to prove Lemma \ref{lemma3}.

\begin{proof}[Proof of Lemma \ref{lemma3}]
We choose without loss of generality $a=1$ and $i=2$. Let $Y$ be the number of neurons belonging to the message $m^1$ which are connected to $(1,2)$. To determine the distribution of $Y$, we split the event $\lbrace Y=i\rbrace$ into the disjoint events $$\lbrace Y=i, Z_{(1,2)}=n\rbrace,$$ where $Z_{(1,2)}$ is the number of messages containing $(1,2)$.

Then,
\begin{eqnarray*}
&&{\mathbb{P}}(Y=i\vert Z_{(1,2)}=n)\\&=&\left(1-\left(1-\frac{1}{l}\right)^n\right)^{i}\left(1-\frac{1}{l}\right)^{n(c-1-i)}.
\end{eqnarray*}
We arrive at

\begin{eqnarray*}
&&{\mathbb{P}}(Y=i)\\
&=&\binom{c-1}{i}\sum_{n=0}^{M-1}\binom{M-1}{n}\frac{1}{l^n}\left(1-\frac{1}{l}\right)^{M-1-n}\\
&&\quad \times \left(1-\left(1-\frac{1}{l}\right)^n\right)^{i}\left(1-\frac{1}{l}\right)^{n(c-1-i)}\\
&=&\binom{c-1}{i}\sum_{n=0}^{M-1}\binom{M-1}{n}\frac{1}{l^n}\\ 
&&\times\left(1-\frac{1}{l}\right)^{M-1-n+n(c-1-i)}\\
&&\times \sum_{k=0}^{i}\binom{i}{k}(-1)^k\left(1-\frac{1}{l}\right)^{nk}\\
&=&\binom{c-1}{i}\sum_{k=0}^{i}\binom{i}{k}(-1)^k\sum_{n=0}^{M-1}\\ 
&&\hskip-.3cm\times  \binom{M-1}{n}\frac{1}{l^n}\left(1-\frac{1}{l}\right)^{M-1-n+nk+n(c-1-i)}\\
&=&\binom{c-1}{i}\sum_{k=0}^{i}\binom{i}{k}(-1)^k\\ &&\times \left(\frac{1}{l}\left(1-\frac{1}{l}\right)^{k+c-1-i}+1-\frac{1}{l}\right)^{M-1}\\
&=&\binom{c-1}{i}\sum_{k=0}^{i}\binom{i}{k}(-1)^k\\
\end{eqnarray*}
\begin{eqnarray*}
&&\times \left(\frac{1}{l}-\frac{k+c-1-i}{l^2}+O(l^{-3})+1-\frac{1}{l}\right)^{M-1}\\
&=&\binom{c-1}{i}\sum_{k=0}^{i}\binom{i}{k}(-1)^k\\ && \times \left(1-\frac{k+c-1-i}{l^2}+O(l^{-3})\right)^{M-1}\\
&=&\binom{c-1}{i}\sum_{k=0}^{i}\binom{i}{k}(-1)^k\\ && \times e^{-\alpha(k+c-1-i)}(1+o(1))\\
&=&\binom{c-1}{i}\left(1-e^{-\alpha}\right)^{i}e^{-\alpha(c-1-i)}(1+o(1)).
\end{eqnarray*}

In particular
$$
{\mathbb{P}}(Y=c-1)=\left(1-e^{-\alpha}\right)^{c-1}(1+o(1)).
$$
\end{proof}

\bibliographystyle{abbrv}

\bibliography{LiteraturDatenbank}
\end{document}